\documentclass[11pt]{amsart}
\usepackage{amsmath}
\usepackage{amssymb}
\usepackage{graphicx}

\newtheorem{theorem}{Theorem}[section]

\newtheorem{lemma}[theorem]{Lemma}

\theoremstyle{definition}

\newtheorem{remark}[theorem]{Remark}

\numberwithin{equation}{section}

\newcommand{\kl}{{\mathcal K}_\ell}
\newcommand{\kln}{{\mathcal K}_{\ell,N}}

\newcommand{\fkl}{\cF_\ell|_{{\mathcal K}_\ell}}
\newcommand{\fkln}{f_\ell|_{{\mathcal K}_{\ell,N}}}
\newcommand{\ws}{\widetilde\Sigma}
\newcommand{\cF}{\mathcal{F}}
\newcommand{\cM}{\mathcal{M}}
\newcommand{\cS}{\mathcal{S}}
\newcommand{\IR}{{\mathbb R}}

\begin{document}

\title{Topological entropy of Bunimovich stadium billiards}

\author[M.~Misiurewicz]{Micha{\l}~Misiurewicz}

\address[Micha{\l}~Misiurewicz]
{Department of Mathematical Sciences\\ Indiana Univer\-sity-Purdue
University Indianapolis\\ 402 N. Blackford Street\\
Indianapolis, IN 46202\\ USA}
\email{mmisiure@math.iupui.edu}

\author[H.-K.~Zhang]{Hong-Kun~Zhang}

\address[Hong-Kun~Zhang]
{Department of Mathematics and Statistics\\ University of
Massachusetts\\ Amherst, MA 01003\\ USA}
\email{hongkun@math.umass.edu}



\date{January 28, 2020}

\thanks{Research of Micha{\l} Misiurewicz was partially
supported by grant number 426602 from the Simons Foundation.}

\begin{abstract}
We estimate from below the topological entropy of the Bunimovich
stadium billiards. We do it for long billiard tables, and find the
limit of estimates as the length goes to infinity.
\end{abstract}

\maketitle

\section{Introduction}

In this paper, we consider Bunimovich stadium billiards. This was
the first type of billiards having convex (focusing) components of the
boundary $\partial \Omega$, yet enjoying the hyperbolic behavior
\cite{Bun74,Bun79}. Such boundary consists of two semicircles at the
ends, joined by segments of straight lines (see Figure~\ref{fig0}).
For those billiards, ergodicity, K-mixing and Bernoulli
property were proved in \cite{CT} for the natural measure.

\begin{figure}[h]
\begin{center}
\includegraphics[width=100truemm]{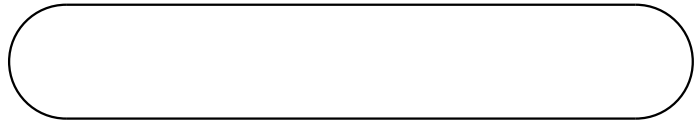}
\caption{Bunimovich stadium.}\label{fig0}
\end{center}
\end{figure}

We consider billiard maps (not the flow) for two-dimensional billiard
tables. Thus, the phase space of a billiard is the product of the
boundary of the billiard table and the interval $[-\pi/2,\pi/2]$ of
angles of reflection. We will use the variables $(r,\varphi)$, where $r$
parametrizes the table boundary by the arc length, and $\varphi$ is the
angle of reflection. We mentioned the natural measure; it is
$c\cos\varphi\;dr\;d\varphi$, where $c$ is the normalizing constant. This
measure is invariant for the billiard map.

As we said, we want to study topological entropy of the billiard map.
This means that we should look at the billiard as a topological
dynamical system. However, existence of the natural measure resulted
in most authors looking at the billiard as a measure preserving
transformation. That is, all important properties of the billiard were
proved only almost everywhere, not everywhere. Additionally, the
billiard map is only piecewise continuous instead of continuous. Often
it is even not defined everywhere. All this creates problems already
at the level of definitions. We will discuss those problems in the
next section.

In view of this complicated situation, we will not try to produce a
comprehensive theory of the Bunimovich stadium billiards from the
topological point of view, but present the results on their
topological entropy that are independent of the approach. For this we
will find a subspace of the phase space that is compact and invariant,
and on which the billiard map is continuous. We will find the
topological entropy restricted to this subspace. This entropy is a
lower bound of the topological entropy of the full system, no matter
how this entropy is defined. Finally, we will find the limit of our
estimates as the length of the billiard table goes to infinity.

The reader who wants to learn more on other properties of the
Bunimovich stadium billiards, can find it in many papers, in
particular \cite{BC, BK, Bun74, Bun79, Ch, CM, HC}. While some of them
contain results about topological entropy of those billiards, none of
those results can be considered completely rigorous.

The paper is organized as follows. In Section~\ref{sec-teb} we discuss
the problems connected with defining topological entropy for
billiards. In Section~\ref{sec-sac} we produce symbolic systems
connected with the Bunimovich billiards. In Section~\ref{sec-cote} we
perform actual computations of the topological entropy.

\section{Topological entropy of billiards}\label{sec-teb}

Let $\cM=\partial\Omega\times[-\pi/2,\pi/2]$ be the phase space of a billiard
and let $\cF:\cM\to\cM$ be the billiard map. We assume that the
boundary of the billiard table is piecewise $C^2$ with finite number
of pieces. In such a situation the map $\cF$ is piecewise continuous
(in fact, piecewise smooth) with finitely many pieces. That is, $\cM$
is the union of finitely many open sets $\cM_i$ (of quite regular
shape) and a singular set $\cS$, which is the union of finitely many
smooth curves, and on which the map is often even not defined. The map
$\cF$ restricted to each $\cM_i$ is a diffeomorphism onto its image.

This situation is very similar as for piecewise continuous piecewise
monotone interval maps. For those maps, the usual way of investigating
them from the topological point of view is to use \emph{coding}. We
produce the \emph{symbolic system} associated with our map by taking
sequences of symbols (numbers enumerating pieces of continuity)
according to the number of the piece to which the $n$-th image of our
point belongs. On this symbolic space we have the shift to the left.
In particular, the topological entropy of this symbolic system was
shown to be equal to the usual Bowen's entropy of the underlying
interval map (see~\cite{MZ}).

Thus, it is a natural idea to do the same for billiards. Thus, for a
point $x\in\cM$, whose trajectory is disjoint from $\cS$, we take its
\emph{itinerary} (code) $\omega(x)=(\omega_n)$, where $\omega_n=i$ if
and only if
$\cF(x)\in\cM_i$. The problem is that the set of itineraries obtained
in such a way is usually not closed (in the product topology). Therefore
we have to take the closure of this set. Then the question one has to
deal with is whether there is no essential dynamics (for example,
invariant measures with positive entropy) on this extra set. A
rigorous approach for coding, including the definition of topological
entropy and a proof of a theorem analogous to the one from~\cite{MZ},
can be found in the recent paper of Baladi and Demers~\cite{BD} about
Sinai billiards.

The Sinai billiard maps are simpler for
coding than the Bunimovich stadium maps. There are finitely many
obstacles on the torus, so the pieces of the boundary, used for the
coding, are pairwise disjoint. This property is
not shared by the Bunimovich stadium billiards. The stadium billiard
is hyperbolic, but not uniformly. Moreover, here we have to
deal with the trajectories that are bouncing between the straight line
segments of the boundary. To complete the list of problems, the coding
with four pieces of the boundary seems to be not sufficient (as has
been noticed in~\cite{BK}).

The papers dealing with the topological entropy of Bunimovich stadium
billiards use different definitions. In~\cite{BK} and~\cite{HC},
topological entropy is explicitly definied as the exponential growth
rate of the number of
periodic orbits of a given period. In~\cite{Ch}, first coding is
performed in a different way, using rectangles defines by stable and
unstable manifolds. This coding uses an infinite alphabet. Then
various definitions of topological entropy for the obtained symbolic
system are used. In~\cite{BD}, topological entropy is defined as the
topological entropy of the corresponding symbolic system, that is, as
the exponential growth rate of the number of nonempty cylinders of a
given length in the symbolic system. As we
mentioned, it is shown that the result is the same as when one is
using the classical Bowen's definition for the original billiard map.
In~\cite{BC}, topological entropy is not formally defined, but it
seems that the authors think of the entropy of the symbolic system.

In this paper, we will be considering a subsystem of the full billiard
map. This subsystem is a continuous map of a compact space to itself,
and is conjugate to a subshift of finite type. Thus, whether we define
the topological entropy of the full system as the entropy of the
symbolic system or as the growth rate of the number of periodic orbit,
our estimates will be always lower bounds for the topological entropy.

\section{Subsystem and coding}\label{sec-sac}

We consider the Bunimovich stadium billiard table, with the radius of
the semicircles 1, and the lengths of straight segments $\ell>1$. The
phase space of this billiard map will be denoted by $\cM_\ell$, and
the map by $\cF_\ell$. The subspace of $\cM_\ell$ consisting of points
whose trajectories have no two consecutive collisions with the same
semicircle will be denoted by $\kl$. The subspace of $\kl$ consisting
of points whose trajectories have no $N+1$ consecutive collisions with
the straight segments will be denoted by $\kln$. We will show that if
$\ell>2N+2$, then the map $\cF_\ell$ restricted to $\kln$ has very
good properties.

In general, coding for $\cF_\ell$ needs at least six symbols. They
correspond to the four pieces of the boundary of the stadium, and
additionally on the semicircles we have to specify the orientation of
the trajectory
(whether $\varphi$ is positive or negative), see~\cite{BK}. However, in
$\kl$ this additional requirement is unnecessary, because there are no
multiple consecutive collisions with the same semicircle. This also
implies that in $\kl$ for a given $\ell$ the angle $\varphi$ is uniformly
bounded away from $\pm\pi/2$.

While in~\cite{BC} the statements about
generating partition are written in terms of measure preserving
transformations, the sets of measure zero that have to be removed are
specified. In $\kl$ the only set that needs to be removed is the set
of points whose trajectories are periodic of period 2, bouncing from
the two straight line segments. However, this set carries no
topological entropy, so we can ignore it. Thus, according
to~\cite{BC}, the symbolic system corresponding to $\cF_\ell$ on $\kl$
is a closed subshift $\Sigma_\ell$ of a subshift of finite type with 4
symbols. We say that there is a \emph{transition} from a state $i$ to
$j$ if it is possible that $\omega_n=i$ and $\omega_{n+1}=j$. In our
subshift here are some transitions that are forbidden: one cannot go
from a symbol corresponding to a semicircle to the same symbol. There
are of course also some transitions in many steps forbidden; they
depend on $\ell$.

For every element of $\Sigma_\ell$ there is a unique point of $\kl$
with that itinerary. However, the same point of $\kl$ may have more than one
itinerary, because there are four points on the boundary of the stadium
that belong to two pieces of the boundary each. Thus, the coding is
not one-to-one, but this is unavoidable if we want to obtain a compact
symbolic system. Another solution would be to remove codes of all
trajectories that pass through any of four special points, and at the
end take the closure of the symbolic space.

This problem disappears when we pass to $\kln$ with $\ell>2N+2$.
Namely, then the angle $\varphi$ at any point of $\kln$ whose first
coordinate is on the straight line piece, is larger than $\pi/4$ in
absolute value.

Let us look at the geometry of this situation. Let $C$ be the right
unit semicircle in $\IR^2$ (without endpoints), $A\in C$, and let
$L_1,L_2$ be half-lines emerging from $A$, reflecting from $C$ (like a
billiard flow trajectory) from inside at $A$ (see Figure~\ref{fig3}).
Assume that for $i=1,2$ the angles between $L_i$ and the horizontal
lines are less than $\pi/4$, and that $L_i$ intersects $C$ only at
$A$. Consider the argument $\arg(A)$ of $A$ (as in polar coordinates
on in the complex plane).

\begin{figure}[h]
\begin{center}
\includegraphics[width=40truemm]{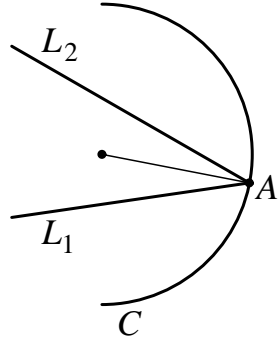}
\caption{Situation from Lemma~\ref{geom}.}\label{fig3}
\end{center}
\end{figure}

\begin{lemma}\label{geom}
In the above situation, $|\arg(A)|<\pi/4$. Moreover, neither $L_1$ nor
$L_2$ passes through an endpoint of $C$.
\end{lemma}

\begin{proof}
If $|\arg(A)|\ge\pi/4$, then both lines $L_1$ and $L_2$ are on the
same side of the origin, so the incidence and reflection angle cannot
be the same. Therefore, $|\arg(A)|<\pi/4$.

Suppose that $L_1$ passes through the lower endpoint of $C$ (the other
cases are similar). Then $\arg(A)<0$, so $L_2$ intersects the
semicircle also at the point with argument
\[
\arg(A)+(\arg(A)-(-\pi/2))=2\arg(A)+\pi/2,
\]
a contradiction.
\end{proof}

In view of the above lemma, the collision points on the semicircles
cannot be too close to the endpoints of the semicircles (including
endpoints themselves). Thus, the
correspondence between $\kln$ and its coding system $\Sigma_{\ell,N}$
is a bijection. Standard considerations of topologies in both systems
show that this bijection is a homeomorphism, say
$\xi_{\ell,N}:\kln\to\Sigma_{\ell,N}$. If $\sigma$ is the left shift
in the symbolic system, then by the construction we have
$\xi_{\ell,N}\circ\cF_\ell=\sigma\circ\xi_{\ell,N}$. In such a way we
get the following lemma.

\begin{lemma}\label{conj}
If $\ell>2N+2$ then the systems $(\kln,\cF_\ell)$ and
$(\Sigma_{\ell,N},\sigma)$ are topologically conjugate.
\end{lemma}

We can modify our codings, in order to simplify further proofs. The
first thing is to identify the symbols corresponding to two
semicircles. This can be done due to the symmetry, and will result in
producing symbolic systems $\Sigma'_\ell$ and $\Sigma'_{\ell,N}$,
which are 2-to-1 factors of $\Sigma_\ell$ and $\Sigma_{\ell,N}$
respectively. Since the operation of taking a 2-to-1 factor preserves
topological entropy, this will not affect our results.

With this simplification, $\Sigma'_\ell$ is a closed, shift-invariant
subset of the phase space of a subshift of finite type $\ws$. Subshift
of finite type $\ws$ looks as follows. There are three states, $0,A,B$
(where 0 corresponds to the semicircles), and the only forbidden
transitions are from $A$ to $A$ and from $B$ to $B$.

Then $\Sigma'_{\ell,N}$ is a closed, shift-invariant subset of
$\Sigma'_\ell$, where additionally $n$-step transitions involving only
states $A$ and $B$ are forbidden if $n>N$. However, it pays to recode
$\Sigma'_{\ell,N}$. Namely, we replace states $A$ and $B$ by
$1,2,\dots,N$ and $-1,-2,\dots,-N$ respectively. If
$(\omega_n)\in\Sigma'_{\ell,N}$, and $\omega_k=\omega_{k+m+1}=0$, while
$\omega_n\in\{A,B\}$ for $n=k+1,k+2,\dots,k+m$, then for the recoded
sequence $(\rho_n)$ we have $\rho_k=\rho_{k+m+1}=0$ and
$(\rho_{k+1},\rho_{k+2}\dots,\rho_{k+m})$ is equal to $(1,2,\dots,m)$
if $\omega_{k+1}=A$ and $(-1,-2,\dots,-m)$ if $\omega_{k+1}=B$.

Geometric meaning of the recoding is simple. We unfold the
stadium by using reflections from the straight parts (see
Figure~\ref{fig1}). We will label the levels of the semicircles by
integers.
\begin{figure}[h]
\begin{center}
\includegraphics[width=100truemm]{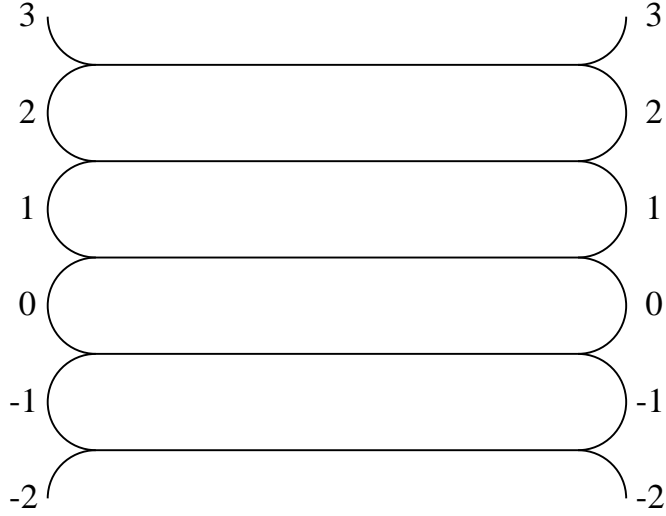}
\caption{Unfolded stadium.}\label{fig1}
\end{center}
\end{figure}
Our new coding translates to this picture as follows. We start at a
semicircle, then go to a semicircle on the other side and $m$ levels
up or down, etc.

For symbolic systems, recoding in such a way amounts to the
topological conjugacy of the original and recoded systems (see~\cite{K}).
This means that the system $(\Sigma'_{\ell,N},\sigma)$ is
topologically conjugate to a subsystem of $\ws_N$, which is the
subshift of finite type defined as follows. The states are
$-N,-N+1,\dots,N-1,N$, and the transitions are: from 0 to every state,
from $i$ to $i+1$ and 0 if $1\le i\le N-1$, from $N$ to 0, from $-i$
to $-i-1$ and 0 if $1\le i\le N-1$, and from $-N$ to 0.

\begin{lemma}\label{level-c}
If $\ell>2N+2$ then $(\Sigma'_{\ell,N},\sigma)$ is topologically
conjugate to $(\ws_N,\sigma)$.
\end{lemma}

\begin{proof}
Both sets $\Sigma'_{\ell,N}$ and $\ws_N$ are closed and
$\Sigma'_{\ell,N}\subset\ws_N$. Therefore, it is enough to prove that
$\Sigma'_{\ell,N}$ is dense in $\ws_N$. For this we show that for
every sequence $(\rho_0,\rho_1,\dots,\rho_k)$ appearing as a block in
an element of $\ws_N$ there is a point $(r_0,\varphi_0)\in\kln$ for which
after coding and recoding a piece of trajectory of length $k+1$, we
get $(\rho_0,\rho_1,\dots,\rho_k)$. By taking a longer sequence, we
may assume that $\rho_0=\rho_k=0$.

Consider all candidates for such trajectories in the unfolded stadium,
when we do not care whether the incidence and reflection angles are
equal. That is, we consider all curves that are unions of straight
line segments from $x_0$ to $x_1$ to $x_2$ $\dots$ to $x_k$ in the
unfolded stadium, such that $x_0$ is in the left semicircle at level
0, $x_1$ is in the right semicircle at level $n_1$, $x_2$ is in the
left semicircle at level $n_1+n_2$, etc. Here $n_1,n_2,\dots$ are the
numbers of non-zero elements of the sequence
$(\rho_0,\rho_1,\dots,\rho_k)$ between a zero element and the next
zero element, where we also take into account the signs of those
non-zero elements. In other words, this curve is an approximate
trajectory (of the flow) in the unfolded stadium that would have the
recoded itinerary $(\rho_0,\rho_1,\dots,\rho_k)$. Additionally we require that
$x_0$ and $x_k$ are at the midpoints of their semicircles. The class
of such curves is a compact space with the natural topology, so there
is the longest curve in this class. We claim that this curve is a
piece of the flow trajectory corresponding to the trajectory we are
looking for.

If we look at the ellipse with foci at $x_i$ and $x_{i+2}$ to which
$x_{i+1}$ belongs, then $x_{i+1}$ has to be a point of tangency of
that elipse and the semicircle. Since for the ellipse the angles of
incidence and reflection are equal, the same is true for the
semicircle.

Now we have to prove three properties of our curve. The first one is that
any small movement of one of the points $x_1,\dots,x_{m-1}$ gives us a
shorter curve. The second one is that none of those points lies at an
endpoint of a semicircle. The third one is that none of the segments
of the curve intersects any semicircle at any other point.

The first property follows from the fact that any ellipse with foci on
the union of the left semicircles at levels $-N$ through $N$, which is
tangent to any right semicircle, is tangent from outside. This is
equivalent to the fact that the maximal curvature of such ellipse is
smaller than the curvature of the semicircles (which is 1). The
distance between the foci of our ellipse is not larger than $2(2N+1)$,
and the length of the large semi-axis is larger than $\ell$.
Elementary computations show that the maximal curvature of such
ellipse is smaller than $\frac{\ell}{\ell^2-(2N+1)^2}$. Thus, this
property is satisfied if $\ell^2-\ell>(2N+1)^2$. However, by the
assumption, $\ell^2-\ell=\ell(\ell-1)\ge(2N+2)(2N+1)>(2N+1)^2$.

The second property is clearly satisfied, because if $x_i$ lies at an
endpoint of a semicircle, then an infinitesimally small movement of
this point along the semicircle would result in both straight segments
of the curve that end/begin at $x_i$ to get longer.

The third property follows from the observation that if $\ell\ge 2N+2$
then the angles between the segments of our curve and the straight
parts of the billiard table boundary are smaller than $\pi/4$. Suppose
that the segment from $x_i$ to $x_{i+1}$ intersects the semicircle $C$
to which $x_{i+1}$ belongs at some other point $y$ (see
Figure~\ref{fig2}). Then $x_{i+1}$ and $y$ belong to the same half of
$C$. By the argument with the ellipses, at $x_{i+1}$ the incidence and
reflection angles of our curve are equal. Therefore, the segment from
$x_{i+1}$ to $x_{i+2}$ also intersects $C$ at some other point, so
$x_{i+1}$ should belong to the other half of $C$, a contradiction.
This completes the proof.
\end{proof}

\begin{figure}[h]
\begin{center}
\includegraphics[width=40truemm]{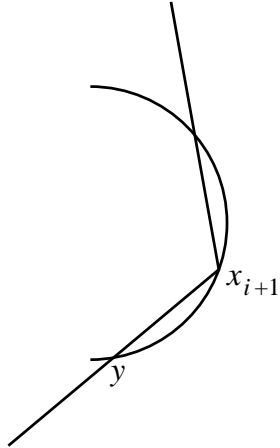}
\caption{Two intersections.}\label{fig2}
\end{center}
\end{figure}

\begin{remark}\label{Markov}
By Lemmas~\ref{conj} and~\ref{level-c} (plus the way we obtained
$\Sigma'_{\ell,N}$ from $\Sigma_{\ell,N}$) it follows that if
$\ell>2N+2$ then the natural partition of $\kln$ into four sets is a
Markov partition.
\end{remark}

\section{Computation of topological entropy}\label{sec-cote}

In the preceding section we obtained some subshifts of finite type. Now
we have to compute their topological entropies. If the alphabet
of a subshift of finite type is $\{1,2,\dots,k\}$, then we
can write the \emph{transition matrix} $M=(m_{ij})_{i,j=1}^n$, where
$m_{ij}=1$ if there is a transition from $i$ to $j$ and $m_{ij}=0$
otherwise. Then the topological entropy of our subshift is the
logarithm of the spectral radius of $M$ (see~\cite{K, ALM}).

\begin{lemma}\label{ent-kl}
Topological entropy of the system $(\Sigma_\ell',\sigma)$ is
$\log(1+\sqrt2)$.
\end{lemma}

\begin{proof}
The transition matrix of $(\Sigma_\ell',\sigma)$ is
\[
\begin{bmatrix}
  1&1&1\\
  1&1&0\\
  1&0&1\\
\end{bmatrix}.
\]
The characteristic polynomial of this matrix is $(1-x)(x^2-2x-1)$, so
the entropy is $\log(1+\sqrt2)$.
\end{proof}

In the case of larger, but not too complicated, matrices, in order to
compute the spectral radius one can use the \emph{rome method}
(see~\cite{BGMY, ALM}). For the transition matrices of $\ws_N$ this
method is especially simple. Namely, if we look at the paths given by
transitions, we see that 0 is a rome: all paths lead to it. Then we
only have to identify the lengths of all paths from 0 to 0 that do not
go through 0 except at the beginning and the end. The spectral radius
of the transition matrix is then the largest zero of the function
$\sum x^{-p_i}-1$, where the sum is over all such paths and $p_i$ is
the length if the $i$-th path.

\begin{lemma}\label{ent-kln}
Topological entropy of the system $(\ws_N,\sigma)$ is the logarithm of
the largest root of the equation
\begin{equation}\label{eq0}
(x^2-2x-1)=-2x^{1-N}.
\end{equation}
\end{lemma}

\begin{proof}
The paths that we mentioned before the lemma, are: one path of length
1 (from 0 directly to itself), and two paths of length $2,3,\dots,N$
each. Therefore, our entropy is the logarithm of the largest zero of
the function $2(x^{-N}+\dots+x^{-3}+x^{-2})+x^{-1}-1$. We have
\[
x(1-x)\big(2(x^{-N}+\dots +x^{-3}+x^{-2})+x^{-1}-1\big)=
(x^2-2x-1)+2x^{1-N},
\]
so our entropy is the logarithm of the largest root of
equation~\eqref{eq0}.
\end{proof}

Now that we computed topological entropies of the subshifts of
finite type involved, we have to go back to the definition of the
topological entropy of billiards (and their subsystems). As we
mentioned earlier, the most popular definitions either employ the
symbolic systems or use the growth rate of the number of periodic
orbits of the given period. For subshifts of finite type that does not
make difference, because the exponential growth rate of the number of
periodic orbits of a given period is the same as the topological
entropy (if the systems are topologically mixing, which is the case
here). As the first step, we get the following result, that follows
immediately from Lemmas~\ref{conj}, \ref{level-c} and~\ref{ent-kln}.

\begin{theorem}\label{t-ent-kln}
If $\ell>2N+2$ then topological entropy of the system
$(\kln,\cF_\ell)$ is the logarithm of the largest root of the
equation~\eqref{eq0}.
\end{theorem}

Now, independently of which definition of the entropy $h(\fkl)$ of
$(\kl,\cF_\ell)$ we choose, we get the next theorem.

\begin{theorem}\label{main}
We have
\begin{equation}\label{eq1}
\liminf_{\ell\to\infty}h(\fkl)\ge 1+\sqrt2.
\end{equation}
\end{theorem}

\begin{proof}
On one hand, $\kln$ is a subset of $\kl$, so $h(\fkl)\ge h(\fkln)$ for
every $N$. Therefore, by Theorem~\ref{t-ent-kln},
\[
\liminf_{\ell\to\infty}h(\fkl)\ge\lim_{N\to\infty}\log y_N,
\]
where $y_N$ is the largest root of the equation~\eqref{eq0}. The
largest root of $x^2-2x-1=0$ is $1+\sqrt2$. In its neighborhood the
right-hand side of~\eqref{eq0} goes uniformly to 0 as $N\to\infty$.
Thus, $\lim_{N\to\infty}y_N=1+\sqrt2$, so we get~\eqref{eq1}.
\end{proof}

If we choose the definition of the entropy via the entropy of the
corresponding symbolic system, then, taking into account
Lemma~\ref{ent-kl}, we get a stronger theorem.

\begin{theorem}\label{main1}
We have
\begin{equation}\label{eq2}
\lim_{\ell\to\infty}h(\fkl)=1+\sqrt2.
\end{equation}
\end{theorem}

Of course, the same lower estimates hold for the whole billiard.

\end{document}